\documentclass{amsart}

\usepackage{amssymb,amsmath,color}
\usepackage{amsthm}
\usepackage{array}
\usepackage{multirow}
\usepackage{url}
\usepackage{tikz}
\usepackage[enableskew]{youngtab}
\usepackage{ytableau,varwidth
}
\allowdisplaybreaks

\def\({\left(}
\def\){\right)}

\newtheorem{theorem}{Theorem}[section]
\newtheorem{corollary}[theorem]{Corollary}
\newtheorem{proposition}{Proposition}
\newtheorem{conjecture}{Conjecture}

\newtheorem{lemma}[theorem]{Lemma}

\theoremstyle{definition}

\newtheorem{example}{Example}
\newtheorem{remark}{Remark}

\begin{document}

\title{Hook length and symplectic content in  partitions}

\author{T. Amdeberhan}
\address{Department of Mathematics\\ Tulane University\\ New Orleans, LA 70118, USA}
\email{tamdeber@tulane.edu}

\author{G. E. Andrews}
\address{Department of Mathematics\\ Penn State University\\ University Park, PA 16802, USA}
 \email{gea1@psu.edu} 

\author{C. Ballantine}
\address{Department of Mathematics and Computer Science\\ College of the Holy Cross \\ Worcester, MA 01610, USA \\} 
\email{cballant@holycross.edu}

\begin{abstract} 
The dimension of an  irreducible representation of $GL(n,\mathbb{C})$, $Sp(2n)$, or $SO(n)$ is given by  the respective hook-length and content formulas for the corresponding partition. The first author,  inspired by the Nekrasov-Okounkov formula, conjectured  combinatorial interpretations of analogous expressions involving hook-lengths and symplectic/orthogonal contents. We prove special cases of these conjectures.  In the process, we show that partitions of $n$ with all symplectic contents non-zero are equinumerous with partitions of $n$ into distinct even parts. We also present Beck-type companions to this identity.  In this context, we give the parity of the number of partitions into distinct parts with odd (respectively, even) rank.  We study the connection between the sum of  hook-lengths and the sum of inversions in the binary representation of a partition.  In addition, we introduce a new partition statistic, the $x$-ray list of a partition, and explore its connection with distinct partitions as well as partitions maximally contained in a given staircase partition.  
\end{abstract}

\maketitle

\section{introduction}\label{intro}

\noindent
The interplay between the hook lengths of a partition $\lambda$  and the contents of its cells appears in Stanley's formula for the dimension of the irreducible polynomial representation of the general linear group, $GL(n, \mathbb C)$, indexed by $\lambda$ \cite[(7.106)]{ec2}. There are also analogous formulas  for irreducible polynomial representations of  symplectic and orthogonal groups \cite{CS, EK, S}. 

\smallskip
\noindent
In this article, we prove a variety of results on the combinatorial nature of the expressions involving the symplectic and orthogonal hook-content formulas and certain generalizations thereof. These objects have been conjectured by  the first author \cite{A12}. In order to state the conjectures and our results, we begin by introducing the relevant notations. 

\smallskip
\noindent
A {\it partition}  $\lambda=(\lambda_1,\lambda_2,\dots,\lambda_{\ell})$ of $n\in\mathbb{N}$, denoted $\lambda\vdash n$, is a finite non-increasing sequence  of positive integers, called \textit{parts}, that add up to $n$. 
The {\it size} of $\lambda$, denoted  by $\vert\lambda\vert$,  is the sum of all its parts. The {\it length} of $\lambda$, denoted by $\ell(\lambda)$, is the number of parts of $\lambda$.  As usual, $p(n)$ denotes the number of partitions of $n$. We  denote by  $\mathcal P(n \mid X)$  the set of partitions of $n$ satisfying some condition $X$ and define $p(n\mid X):=|\mathcal P(n\mid X)|$.  The {\it Young diagram} of a partition $\lambda$ is a left-justified array of squares such that the $i^{th}$-row of the array contains $\lambda_i$ squares. 
\begin{example}\label{eg1} If $\lambda= (5, 3, 3, 2,1)$, then $\vert\lambda\vert=14, \ell(\lambda)=5$. The  Young diagram of $\lambda$ is shown below. 
$$\tiny\ydiagram{5, 3, 3, 2,1}$$ \end{example}
\noindent
A cell $(i,j)$ of $\lambda$ is the square in the $i^{th}$-row and $j^{th}$-column in the Young diagram of $\lambda$. The conjugate of $\lambda$ is the partition $\lambda'$ whose Young diagram has rows that are precisely the columns of the Young diagram of $\lambda$. For example, $\lambda'=(6,5,4,2,2)$ is the conjugate of $\lambda= (5,5, 3, 3, 2,1)$. When necessary, we will use the exponential notation $\lambda=(a_1^{m_1},   a_2^{m_2}, \ldots, a_k^{m_k})$ to mean that $a_i$ appears $m_i$ times as a part of $\lambda$, for $i=1, 2, \ldots, k$. This will mainly be used for hooks; that is, partitions of the form $(a, 1^b)$. The {\it hook length} of a cell $u=(i,j)$  of $\lambda$ is $h^\lambda(u)=\lambda_i+\lambda_j'-i-j+1$ and its {\it content} is $c^\lambda(u)=j-i$. Then, the dimension of the irreducible  representation of $GL(n, \mathbb C)$ corresponding to $\lambda$ with $\ell(\lambda)\leq n$ is given by \cite[(7.106)]{ec2} $$\dim_{GL}(\lambda,n)=\prod_{u\in \lambda}\frac{n+c^\lambda(u)}{h^\lambda(u)}.$$ 

\smallskip
\noindent
The irreducible representations of the symplectic group $Sp(2n)$, consisting of  $2n\times 2n$ matrices which preserve any non-degenerate, skew-symmetric form on $\mathbb C^{2n}$, are also indexed by partitions $\lambda$ with $\ell(\lambda)\leq n$. The {\it symplectic content}  of cell $(i,j)\in\lambda$ is
$$c^\lambda _{sp}(i,j)=\begin{cases}\lambda_i+\lambda_j-i-j+2 & \text{ if } i>j\\ i+j-\lambda'_i-\lambda'_j & \text{ if } i\leq j.
\end{cases}$$ 
The irreducible  representations of the special orthogonal group $SO(n)$, consisting of orthogonal matrices of determinant $1$,  are  indexed by partitions $\lambda$ with $\ell(\lambda)\leq \lfloor\frac{n}{2}\rfloor$.
The {\it orthogonal content}  of cell $(i,j)\in\lambda$ is defined by
$$c_{O}^{\lambda}(i,j)=\begin{cases} \lambda_i+\lambda_j-i-j \qquad \text{if $i\geq j$} \\
i+j-\lambda_i'-\lambda_j'-2 \,\,\, \text{if $i<j$}.\end{cases}$$ 
The symplectic and orthogonal counterparts to Stanley's hook-content formula are,  respectively, \cite{EK, S}
$$\dim_{Sp}(\lambda,2n)=\prod_{u\in\lambda}\frac{2n+c_{sp}^{\lambda}(u)}{h^{\lambda}(u)} \qquad \text{and} \qquad
\dim_{SO}(\lambda,n)=\prod_{u\in\lambda}\frac{n+c_{O}^{\lambda}(u)}{h^{\lambda}(u)}.$$ 

\smallskip
\noindent
We are now ready to introduce the conjectures listed in \cite{A12}.  The author was inspired by the remarkable hook-length identity of Nekrasov and Okounkov \cite[(6.12)]{NO06}  (later given
a more elementary proof by Han \cite{H10})
$$\sum_{n\geq0}q^n\sum_{\lambda\vdash n}\prod_{u\in\lambda}\frac{t+(h^{\lambda}(u))^2}{(h^{\lambda}(u))^2}=\prod_{j\geq1}\frac1{(1-q^j)^{t+1}},$$
and  Stanley's  analogous hook-content identity \cite[Theorem 2.2]{RPS10}
$$\sum_{n\geq0}q^n\sum_{\lambda\vdash n}\prod_{u\in\lambda}\frac{t+(c^{\lambda}(u))^2}{(h^{\lambda}(u))^2}=\frac1{(1-q)^t}.$$

\smallskip
\noindent
\begin{conjecture} [\cite{A12}, Conjecture 6.2(a)] \label{conj6.2a} For $t$ an indeterminate, we have 
$$\sum_{n\geq0}q^n\sum_{\lambda\vdash n}\prod_{u\in\lambda}\frac{t+c_{sp}^{\lambda}(u)}{h^{\lambda}(u)}=
\prod_{j\geq1}\frac{(1-q^{8j})^{\binom{t+1}2}}{(1-q^{8j-2})^{\binom{t+1}2-1}}
\left(\frac{1-q^{4j-1}}{1-q^{4j-3}}\right)^t
\left(\frac{1-q^{8j-4}}{1-q^{8j-6}}\right)^{\binom{t}2-1}.$$ \end{conjecture}

\begin{conjecture} [\cite{A12}, Conjecture 6.2(b)] \label{conj6.2b} For $t$ an indeterminate, we have 
$$\sum_{n\geq0}q^n\sum_{\lambda\vdash n}\prod_{u\in\lambda}\frac{t+c_{O}^{\lambda}(u)}{h^{\lambda}(u)}=
\prod_{j\geq1}\frac{(1-q^{8j})^{\binom{t}2}}{(1-q^{8j-6})^{\binom{t}2-1}}
\left(\frac{1-q^{4j-1}}{1-q^{4j-3}}\right)^t
\left(\frac{1-q^{8j-4}}{1-q^{8j-2}}\right)^{\binom{t+1}2-1}.$$ \end{conjecture}

\smallskip
\noindent In this article, we prove the case $t=0$ of Conjectures \ref{conj6.2a} and \ref{conj6.2b}. 

\smallskip
\noindent 
The next conjecture from \cite{A12} is a symplectic, respectively orthogonal, counterpart to the hook-content identity of Stanley, which is also  in the spirit of Nekrasov-Okounkov's hook formula.

\begin{conjecture} [\cite{A12}, Conjecture 6.3(a)] \label{conj6.3a} For $t$ an indeterminate, we have 
$$\sum_{n\geq0}q^n\sum_{\lambda\vdash n}\prod_{u\in\lambda}\frac{t+(c_{sp}^{\lambda}(u))^2}{(h^{\lambda}(u))^2}=
\prod_{j\geq1}\frac1{(1-q^{4j-2})(1-q^j)^t}=
\sum_{n\geq0}q^n\sum_{\lambda\vdash n}\prod_{u\in\lambda}\frac{t+(c_{O}^{\lambda}(u))^2}{(h^{\lambda}(u))^2}.$$ \end{conjecture}

\smallskip
\noindent We  prove the cases $t=0$ and  $t=-1$ of Conjecture \ref{conj6.3a}.

\smallskip
\noindent The special cases of the above conjectures will be proved in section \ref{no-thms}. We denote by $sy\mathcal P_0$ the set of  partitions $\lambda$ such that $c^{\lambda}_{sp}(u)\neq 0$ for all cells $u\in \lambda$. To prove the case $t=0$ of Conjectures \ref{conj6.2a}, \ref{conj6.2b} and \ref{conj6.3a},  we describe the partitions in $sy\mathcal P_0$ explicitly. 
The {\it nested hooks} of a partition $\lambda$ are hook  partitions whose Young diagrams consist of a cell $(i,i)$ on the main diagonal of $\lambda$ together with all cells directly to the right and directly below the cell $(i,i)$. Thus, the $i^{th}$-hook of $\lambda$ is the partition $(\lambda_i-i+1, 1^{\lambda'_i-i})$. In section \ref{no-thms} we prove the following characterization. 

\begin{theorem} \label{non-zero cont} $\lambda\in sy\mathcal P_0$ if and only if $\lambda=\emptyset$ or all nested hooks $(a, 1^b)$ of $\lambda$ satisfy $a=b$. 

\end{theorem}
\noindent Using the transformation of {\it straightening} nested hooks, i.e., transforming each nested hook  $(a, 1^b)$ into a part equal to $a+b$, one can easily see that partitions of $n$ in $sy\mathcal P_0$ are in bijection with partitions of $n$ into distinct even parts.  Therefore \begin{equation}\label{dist-parts} |sy\mathcal P_0(n)|=p(n \, \big| \text{ distinct even parts}). \end{equation}

\smallskip
\noindent In section \ref{beck-section}, we prove two Beck-type companion identities to \eqref{dist-parts}.  These would give combinatorial interpretations for the excess of the number of parts in all partitions in $sy\mathcal P_0(n)$ over the number of parts in all partitions in $\mathcal P(n \, \big| \text{ distinct even parts})$. The combinatorial description in our first Beck-type identity  involves partitions with even rank. In section \ref{odd-even rank}, we examine the parity of $p(n\mid \text{distinct parts, odd rank})$ and $p(n\mid \text{distinct parts, even rank})$. In section \ref{bi-hooks}, we will investigate the connection between the sum of all hook-lengths of a partition $\lambda$ and the sum of all inversions in the binary representation of $\lambda$.  Finally, in section \ref{two stats}, we study the $x$-ray list of a partition, an analogue to the  $x$-ray of a permutation \cite{BMPS}, and we determine links with $p(n \mid \text{distinct parts})$ and also with the number of partitions maximally contained in a given staircase partition.

\section{Proofs of Theorem  \ref{non-zero cont} and special cases of  Conjectures \ref{conj6.2a}, \ref{conj6.2b} and \ref{conj6.3a}}\label{no-thms}

\noindent In this section, we first prove Theorem \ref{non-zero cont} and derive two corollaries which lead to the proofs for the case $t=0$ of Conjectures  \ref{conj6.2a}, \ref{conj6.2b} and \ref{conj6.3a}. We then conclude by proving the case $t=-1$ of Conjecture \ref{conj6.3a}. 

\smallskip
\noindent We start with basic observations about the symplectic and orthogonal contents of cells in partitions. 

\begin{remark} \label{remark sp-o} The definitions imply that, for any partition $\lambda$ and any cell $(i,j)\in \lambda$, we have $$c^\lambda_{sp}(i,j)=-c^{\lambda'}_O(j,i).$$

\smallskip
\noindent  Moreover, since $h^\lambda(i,j)=-h^{\lambda'}(j,i)$, for any positive integer $n$, we have 
$$\sum_{\lambda\vdash n}\prod_{u\in\lambda}\frac{c_{sp}^{\lambda}(u)}{h^{\lambda}(u)}=
\sum_{\lambda\vdash n}\prod_{u\in\lambda}\frac{c_{O}^{\lambda}(u)}{h^{\lambda}(u)}.$$ 

\end{remark}

\smallskip
\noindent Given a partition $\lambda$, by the {\it outer hook} of $\lambda$ we mean the hook determined by the cell $(1,1)$, i.e., the partition  $(\lambda_1, 1^{\lambda'_1-1})$.

\begin{remark} \label{R1}  Let $\mu$ be the partition obtained from $\lambda$ by removing the outer hook.  Then, it is immediate from the definitions that 
$$c^\mu _{sp}(i,j)=c^\lambda _{sp}(i+1,j+1) \qquad \text{and} \qquad c^\mu _{O}(i,j)=c^\lambda _{O}(i+1,j+1).$$
Thus, the symplectic and orthogonal contents of a cell are preserved by the operation of adding or removing an outer hook.
\end{remark}

\smallskip
\noindent The \textit{rank}, $r(\lambda)$,  of a partition $\lambda$ is defined as the difference between the largest part and the length of the partition, i.e., $r(\lambda)=\lambda_1-\lambda'_1$. We denote by $r_j(\lambda)$ the rank of the partition obtained from $\lambda$ by removing the first $j-1$ outer hooks, successively. Thus, $r_j(\lambda)=\lambda_j-\lambda'_j$.

\begin{remark}\label{remark c-h}  For any partition $\lambda$ and any cell $(i,j)\in \lambda$, we have $$c^\lambda_{sp}(i,j)=\begin{cases} r_j(\lambda)+1+h^\lambda(i,j) & \text{ if } i>j\\ r_i(\lambda)+1-h^\lambda(i,j) & \text{ if } i\leq j.\end{cases}  $$
\end{remark}

\smallskip
\noindent Next, we determine explicitly the symplectic contents of cells in hook partitions $\lambda=(a, 1^b)$ with $a\geq 1, b\geq 0$. In detail, we have
\begin{align*}c^\lambda_{sp}(1,1)& =-2b, & \  \\ c^\lambda_{sp}(1,j)&=-b+j-1 & \text{ for } &2\leq j\leq a,\\ c^\lambda_{sp}(i,1)&=a-i+2 & \text{ for } &2\leq i\leq b+1.
\end{align*} The contents $c^\lambda_{sp}(1,j)$, for $j\geq 2$, are seen as consecutive integers between $-b+1$ and $a-b-1$, increasing from left to right. The contents 
$c^\lambda_{sp}(i,1)$, for $i\geq 2$, are consecutive integers between $a-b+1$ and $a$ decreasing from top to bottom. These observations lead to the following characterization of hook partitions when all symplectic contents are unequal to a fixed integer $t$. 

\newpage

\begin{proposition}\label{not t} Let $t$ be an integer. A hook partition $\lambda=(a, 1^b)$, with $a\geq 1, b\geq 0$, satisfies $c^\lambda_{sp}(i,j)\neq t$ for all cells $(i,j)\in \lambda$ if and only if  all of  the following conditions is satisfied: 
\begin{itemize}
\item[(i)] if $b\neq -t/2$

\item[(ii)] if $a\geq 2$, then $b<1-t$ or $a-b-1<t$;

\item[(iii)] if $b\geq 1$, then $t<a-b+1$ or $a<t$.
\end{itemize}
\end{proposition}
\noindent The cases needed for our proofs are given in the next two corollaries. Recall that $sy\mathcal P_0$ denotes the set of partitions having all symplectic contents non-zero. We also denote by $sy\mathcal P_{\pm1}$ the set of partitions with all symplectic contents not equal to $\pm 1$. 

\begin{corollary} \label{cor 0} For $a\geq 1, b\geq 0$, the hook partition $\lambda =(a, 1^b)$ is in $sy\mathcal P_0$ if and only if $a=b$. 
\end{corollary}
\begin{proof} Choose $t=0$ in Proposition \ref{not t} and suppose $(a, 1^b)\in sy\mathcal P_0$. Proposition \ref{not t}(i) implies $b\geq 1$. If $a=1$, by Proposition \ref{not t}(iii), we gather $b<2$. Therefore $b=1=a$. If $a\geq 2$, then, by (ii) and (iii) of  Proposition \ref{not t}, $a=b$. Conversely, if $a=b$, then the conditions of Proposition \ref{not t} are satisfied. 
\end{proof}

\begin{corollary} \label{cor 1} The only hook partitions  in $sy\mathcal P_{\pm 1}$ are $(1)$ and $(2,1)$. 
\end{corollary}

\begin{proof} The two partitions $(1)$ and $(2,1)$ satisfy the conditions of Proposition \ref{not t} with $t=\pm 1$. 
Consider $(a, 1^b)$ with $a\geq 1, b\geq 0$ and $(a, 1^b)\not \in \{(1),(2,1)\}$. Condition (i)  is satisfied. If $b=0$, condition (ii) is not satisfied for any $a\geq 2$.  If $a=1$, then condition (iii) is not satisfied for any $b\geq 1$.  Suppose $a\geq 2$ and $b\geq 1$. Conditions (ii) and (iii)   imply $a=b+1$ if $t=1$, and imply $a=b-1$ if $t=-1$. Thus, the only hook partitions in $sy\mathcal P_{\pm 1}$ are $(1)$ and $(2,1)$. 
\end{proof}
\noindent Before we prove Theorem \ref{non-zero cont} we introduce one more concept.  The {\it Durfee square} of a partition $\lambda$ is the largest partition of the form $(m^m)$ whose Young diagram fits inside the Young diagram of $\lambda$. The  length of the Durfee square of $\lambda$  equals the number of nested hooks of $\lambda$. 

\begin{proof}[Proof of Theorem \ref{non-zero cont}] Clearly $\emptyset\in sy\mathcal P_0$. If $\lambda \neq \emptyset$,  we prove by induction on the  length $m$ of the Durfee square that $\lambda \in sy\mathcal P_0$ if and only if  $a=b$. 

\smallskip
\noindent If $m=1$, the statement is true by Corollary \ref{cor 0}. 

\smallskip
\noindent  Next, assume that the statement of the theorem is true for every partition with Durfee square of length less than $k$ and let $\lambda$ be a partition with Durfee square of length $k$. Suppose $\lambda\in sy\mathcal P_0$. By Remark \ref{R1}, we only need to show that the outer hook, $(\lambda_1, 1^{\lambda'_1-1})$, satisfies $\lambda_1=\lambda'_1-1$. 

\smallskip

\noindent 
If $\lambda_1<\lambda'_1-1$,   we must have $\lambda_{\lambda_1+2}=1$. Otherwise, if $\mu$ is the partition obtained from $\lambda$ by removing the outer hook, we have $\mu'_1\geq \lambda_1+1\geq\mu_1+2$ which is impossible since $\mu \in sy\mathcal P_0$ and hence $\mu_1=\mu'_1-1$. 
Consequently,  we see that $c^\lambda_{sp}(\lambda_1+2, 1)=\lambda_{\lambda_1+2}+\lambda_1-(\lambda_1+2)-1+2= 0$.
 Similarly, if $\lambda_1>\lambda'_1-1$, we have  $\lambda'_{\lambda'_1}=1$ and thus $c_{sp}(1, \lambda'_1)=1+\lambda'_1-\lambda'_1-\lambda'_{\lambda'_1}=0$.
Therefore $\lambda_1=\lambda'_1-1$. 

\smallskip

\noindent Finally, if every nested hook $(a, 1^b)$ of $\lambda$ satisfies $a=b$, by induction, $c^\lambda_{sp}(i,j)\neq 0$ for all $i,j\geq 2$. Moreover, since $\lambda_1=\lambda'_1-1$, for $1\leq j\leq \lambda_1$ we have 
$$ c^{\lambda}_{sp}(1,j)=1+j-\lambda'_1-\lambda'_j\leq 1+\lambda_1-\lambda'_1-\lambda'_j=-\lambda'_j<0,$$
and for $2\leq i\leq \lambda'_1$ we have 
$ c^{\lambda}_{sp}(i,1)=\lambda_i+\lambda_1-i-1+2\geq\lambda_i+\lambda_1-\lambda'_1+1 \geq \lambda_i>0$.
Thus, $\lambda \in sy\mathcal P_0$.
\end{proof}

\begin{corollary}\label{Euler} For any positive integer $n$ we have  $$|sy\mathcal P_0(n)|=p(n \, \big| \text{ parts} \equiv 2\bmod 4).$$\end{corollary}
\begin{proof} Apply \eqref{dist-parts} and  Euler's identity  \cite[(1.2.5)]{A98} with $q$ replaced by $q^2$. 
\end{proof}

\smallskip

\noindent From Theorem \ref{non-zero cont} and Remark \ref{remark c-h} we have the following. 

\begin{corollary} \label{cc=h} For $\lambda\neq \emptyset$,  we have $\lambda \in sy\mathcal P_0$ if and only if  $$c^\lambda_{sp}(i,j)=\begin{cases} h^\lambda(i,j) & \text{ if } i>j \\ -h^\lambda(i,j) & \text{ if } i\leq j,\end{cases}  $$ for all $(i,j)\in \lambda$.
\end{corollary}
\begin{proof} The assertion holds due to the fact that  all nested hooks $(a, 1^b)$ of $\lambda$ satisfy $a=b$ if and only if  $r_j(\lambda)=-1$ for all $j$ no larger than the length of the Durfee square of $\lambda$.  \end{proof}

\begin{remark} From Corollary \ref{cc=h},  if $\lambda \in sy\mathcal P_0$, we have $c^\lambda_{sp}(i,j)>0$ if $i> j$ and $c^\lambda_{sp}(i,j)<0$ if $i\leq j$.
\end{remark}
\noindent From Corollaries \ref{Euler} and \ref{cc=h} and Remark \ref{remark sp-o} we obtain the proof for the case $t=0$ of Conjecture \ref{conj6.3a}.

\begin{theorem} [\cite{A12}, Conjecture 6.3(b)] \label{conj6.3b} 
$$\sum_{n\geq0}q^n\sum_{\lambda\vdash n}\prod_{u\in\lambda}\frac{(c_{sp}^{\lambda}(u))^2}{(h^{\lambda}(u))^2}=
\prod_{j\geq1}\frac1{1-q^{4j-2}}=\sum_{n\geq0}q^n\sum_{\lambda\vdash n}\prod_{u\in\lambda}\frac{(c_{O}^{\lambda}(u))^2}{(h^{\lambda}(u))^2}.$$ \end{theorem}

\smallskip
\noindent 
Next, we prove care $t=0$ of Conjectures \ref{conj6.2a} and \ref{conj6.2b}.

\begin{theorem} [\cite{A12}, Conjecture 6.2(c)] \label{conj6.2c} 
$$\sum_{n\geq0}q^n\sum_{\lambda\vdash n}\prod_{u\in\lambda}\frac{c_{sp}^{\lambda}(u)}{h^{\lambda}(u)}=
\prod_{j\geq1}\frac1{1+q^{4j-2}}=\sum_{n\geq0}q^n\sum_{\lambda\vdash n}\prod_{u\in\lambda}\frac{c_{O}^{\lambda}(u)}{h^{\lambda}(u)}.$$ \end{theorem}

\begin{proof}  From Theorem \ref{non-zero cont}, if $n$ is odd,  $sy\mathcal P_0(n)=\emptyset$. Assume $n$ is even. If $\lambda \in sy\mathcal P_0$, by Corollary \ref{cc=h}, for each $u\in \lambda$, we have $\displaystyle \frac{c^\lambda_{sp}(u)}{h^\lambda(u)}=\pm 1$. 
Since all nested hooks $(a, 1^b)$ of $\lambda$ satisfy $a=b$,  the number of cells $u\in \lambda$ with $\displaystyle \frac{c^\lambda_{sp}(u)}{h^\lambda(u)}=1$ equals $\displaystyle \frac{n}{2}$. Thus, if $n>0$ is even and $\lambda \vdash n$, 
$$\prod_{u\in\lambda}\frac{c_{sp}^{\lambda}(u)}{h^{\lambda}(u)}=\begin{cases} 1 & \text{ if } n\equiv 0\bmod 4\\  -1 & \text{ if } n\equiv 2\bmod 4.\end{cases}$$
On the other hand, $\displaystyle \prod_{j\geq 1}\frac{1}{1+q^{4j-2}}$ is the generating function for the number of partitions $\lambda\vdash n$ with parts congruent to $2\bmod 4$, each partition counted with weight $(-1)^{\ell(\lambda)}$. If $\lambda$ is such a partition, then 
$$(-1)^{\ell(\lambda)}=\begin{cases} 1 & \text{ if } n\equiv 0\bmod 4\\  -1 & \text{ if } n\equiv 2\bmod 4.\end{cases}$$ 
This proves the left-hand identity of Theorem \ref{conj6.2c}. The right-hand identity follows from Remark \ref{remark sp-o}. 
\end{proof}

\smallskip
\noindent 
Theorems \ref{conj6.2c} and \ref{conj6.3b} lead to the following identity as it was conjectured in \cite{A12}. 

\begin{corollary} [\cite{A12}, Conjecture 6.3(c)] \label{conj6.3c} For any positive integer $n$ we have
$$(-1)^{\binom{n}2}\sum_{\lambda\vdash n}\prod_{u\in\lambda}\frac{(c_{sp}^{\lambda}(u))^2}{(h^{\lambda}(u))^2}
=\sum_{\lambda\vdash n}\sum_{\lambda\vdash n}\prod_{u\in\lambda}\frac{c_{sp}^{\lambda}(u)}{h^{\lambda}(u)}.$$ \end{corollary}

\noindent
Before proving the case $t=-1$ of Conjecture \ref{conj6.3a}, we characterize the partitions in $sy\mathcal P_{\pm 1}$.

\smallskip\noindent
Denote by $\delta_r$ the {\it staircase partition} with $r$ consecutive parts $\delta_r=(r, r-1, \ldots, 2,1)$. Notice that the length of the Durfee square of $\delta_r$ is $\lceil \frac{r}{2}\rceil$.

\begin{theorem} \label{non-pm1 cont} $\lambda\in sy\mathcal P_{\pm 1}$ if and only if $\lambda=\emptyset$ or $\lambda=\delta_r$ for some $r$. 
\end{theorem}

\begin{proof} Clearly $\emptyset\in sy\mathcal P_{\pm1}$. If $\lambda \neq \emptyset$,  we prove the statement by induction on the  length $m$ of the Durfee square of the partition.  

\smallskip
\noindent If $m=1$, the statement is true by Corollary \ref{cor 1}.

\smallskip
\noindent
For the inductive step, assume that if $\mu$ is a partition with Durfee square of length at most $m$, then $\mu\in sy\mathcal P_{\pm1}$ if and only if $\mu$ is a staircase partition. Let $\lambda\in sy\mathcal P_{\pm 1}$ be a partition with Durfee square of size $m+1$. Then by the induction hypothesis, the partition $\lambda^-$ obtained from $\lambda$ by removing the outer hook is a staircase with Durfee size $m$, i.e., $\lambda^-=\delta_{k}$ with $k\in \{2m, 2m-1\}$.  Conversely, it follows from Remark \ref{R1} that if $\lambda^-=\delta_{k}$ with $k\in \{2m, 2m-1\}$, then  $c^\lambda_{sp}(i,j)\neq \pm 1$ for all cells $(i,j)\in \lambda$ with $i,j\geq 2$. Assume $\lambda$ is such that $\lambda^-=\delta_{k}$ with $k\in \{2m, 2m-1\}$.
We have $$c^\lambda_{sp}(1,j)=\begin{cases}2-2\lambda'_1  & \text{ if } j=1\\ -\lambda'_1-k+2j-2 & \text{ if } 2\leq j\leq k+1\\ j-\lambda'_1 & \text{ if } j\geq k+2 \text{ \ (if any) } \end{cases} $$ and 
$$c^\lambda_{sp}(j,1)=\begin{cases}2-2\lambda'_1  & \text{ if } j=1\\ \lambda_1+k-2j+4 & \text{ if } 2\leq j\leq k+1\\ \lambda_1-j+2 & \text{ if } j\geq k+2 \text{ \ (if any) } \end{cases}$$ By construction, $\lambda_1, \lambda'_1\geq k+1$.

\smallskip \noindent If $\lambda_1=\lambda'_1=k+1$, then $c^\lambda_{sp}(1,k+1)=-1$.

\smallskip \noindent  If $\lambda_1=k+1, \lambda'_1\geq k+2$, then $c^\lambda_{sp}(k+2,1)=1$. 

\smallskip \noindent  If $\lambda_1\geq k+2$ and  $\lambda'_1= k+1$,  then $c^\lambda_{sp}(1,k+2)=1$.

\smallskip \noindent  If $\lambda'_1>\lambda_1\geq k+2$,  then $c^\lambda_{sp}(\lambda_1+1,1)=1$.

\smallskip \noindent  If $\lambda_1\geq\lambda'_1\geq k+3$,  then $c^\lambda_{sp}(1,\lambda'_1-1)=-1$.

\smallskip \noindent Finally, if $\lambda_1=\lambda'_1=k+2$, we have $c^\lambda_{sp}(i,j)$ is even for all cells $(i,j)$. \end{proof}

\noindent From Theorem \ref{non-pm1 cont} and Remark \ref{remark c-h} we have the following. 

\begin{corollary} \label{c=h} For $\lambda\neq \emptyset$,  we have $\lambda \in sy\mathcal P_{\pm1}$ if and only if  $$c^\lambda_{sp}(i,j)=\begin{cases} h^\lambda(i,j)+1 & \text{ if } i>j \\ -h^\lambda(i,j)+1 & \text{ if } i\leq j,\end{cases}  $$ for all $(i,j)\in \lambda$.
\end{corollary}\begin{proof} The statement follows from the fact that  $\lambda$ is a staircase partition if and only if $r_j(\lambda)=0$ for all $j$ no larger than the length if the Durfee square of $\lambda$.  \end{proof}

\begin{theorem}\label{t=-1} We have the generating function
$$\sum_{n\geq0}q^n\sum_{\lambda\vdash n}\prod_{u\in\lambda}\frac{c_{sp}^{\lambda}(u))^2-1}{(h^{\lambda}(u))^2}=\prod_{j\geq1}\frac{1-q^j}{1-q^{4j-2}}=\sum_{n\geq0}(-q)^{\binom{n+1}2}.$$
\end{theorem}

\begin{proof}

\smallskip 
\noindent Suppose $\lambda=\delta_k\vdash n$. For each each $1\leq j\leq k$, we refer to the cells $(a,c)\in \delta_k$ with $a+c=j+1$ as the $j^{th}$ anti-diagonal of $\lambda$. 
Then, in the $j^{th}$ anti-diagonal we have $j$ cells with hook-length $2(k-j)+1$. Of these, $\lceil \frac{j}{2}\rceil$ cells have symplectic content  $-2(k-j)$ and $\lfloor \frac{j}{2}\rfloor$ cells have symplectic content $2(k-j+1)$. Therefore, if $n$ is not a triangular number, $\sum_{\lambda\vdash n} \prod_{u\in\lambda}\frac{c_{sp}^{\lambda}(u))^2-1}{(h^{\lambda}(u))^2}=0$ and if $n$ is the triangular number $\binom{k+1}{2}$, then 

\begin{align*}
\sum_{\lambda\vdash n}&\prod_{u\in\lambda}\frac{c_{sp}^{\lambda}(u))^2-1}{(h^{\lambda}(u))^2} \\
= &\prod_{j=1}^k\frac{(4(k-j)^2-1)^{\lceil\frac{j}{2}\rceil}(4(k-j+1)^2-1)^{\lfloor\frac{j}{2}\rfloor}}{(2(k-j)+1)^{2j}}\\ 
 =& \prod_{j=1}^k\frac{(2(k-j)-1)^{\lceil\frac{j}{2}\rceil}(2(k-j+1)+1)^{\lfloor\frac{j}{2}\rfloor}}{(2(k-j)+1)^{j}}\\ 
 =& \frac{ (-1)^{\lceil\frac{k}{2}\rceil}\prod_{j=2}^k(2(k-j+1)-1)^{\lceil\frac{j-1}{2}\rceil} \cdot(2k-1)
\prod_{j=2}^{k-1}(2(k-j)+1)^{\lfloor\frac{j+1}{2}\rfloor}}{\prod_{j=1}^k(2(k-j)+1)^{j}}.
\end{align*}
Since $\lceil\frac{j-1}{2}\rceil+\lfloor \frac{j+1}{2}\rfloor=j$ for all $2\leq j\leq k-1$, we obtian
$$\sum_{\lambda\vdash n} \prod_{u\in\lambda}\frac{c_{sp}^{\lambda}(u))^2-1}{(h^{\lambda}(u))^2} =(-1)^{\lceil \frac{k}{2}\rceil}=(-1)^{\binom{k+1}{2}}=(-1)^n.$$
Lastly, $$ \prod_{j\geq1}\frac{1-q^j}{1-q^{4j-2}}=\sum_{n\geq0}(-q)^{\binom{n+1}2}$$ follows from Gauss' identity  \cite[(2.2.13)]{A98}, with $q$ replaced by $-q$, and Euler's identity  \cite[(1.2.5)]{A98}, with $q$ replaced by $q^2$. 
\end{proof}

\begin{remark}
Let $\beta(n)$ be the number of {\it overcubic partitions} of an integer $n$, see \cite{K12} and references therein. The generating function for $\beta(n)$ is  
$$\sum_{n\geq0}\beta(n)\,q^n=\prod_{j\geq1}\frac1{(1-q^{4j-2})(1-q^j)^2}=\prod_{j\geq1}\frac{1+q^{2j}}{(1-q^j)^2}.$$ 
Thus, the case $t=2$ of Conjecture \ref{conj6.3a}, which is still an open problem, becomes
$$\beta(n)=\sum_{\lambda\vdash n}\prod_{u\in\lambda}\frac{2+(c_{sp}^{\lambda}(u))^2}{(h^{\lambda}(u))^2}.$$\end{remark}

\smallskip
\noindent
\section {Beck-Type identities}\label{beck-section}

\smallskip
\noindent 
In this section, we consider again the identity  \eqref{dist-parts},
$$p(n\mid \mbox{distinct even parts})=\vert sy\mathcal P_0(n)\vert,$$
and establish two Beck-type companion identities for \eqref{dist-parts}. If $\mathcal P(n \mid X)$ denotes the set of partitions of $n$ satisfying condition $X$ and $p(n\mid  X)=|\mathcal P(n\mid X)|$, a Beck-type companion identity to $p(n\big| X)= p(n\mid Y)$ is an identity that equates the excess of the number of parts of all partitions in $\mathcal P(n\mid X)$ over the number of parts of all partitions in $\mathcal P(n\mid Y)$ and (a multiple of) the number of partitions of $n$ satisfying a condition that is a slight relaxation of $X$ (or $Y$). 

\smallskip\noindent
Recall that $sy\mathcal P_0(n)$ is the set of partitions of $n$ for which the symplectic content of all cells is non-zero. From the proof of Theorem  \ref{non-zero cont}, these partitions are precisely those which are {\it almost self-conjugate}, i.e.. the nested hooks have $\text{leg $=$ arm $+1$}$. Moreover, if the Durfee square has side length $m$, then the $(m+1)^{st}$ part of the partition is $m$ and removing a box from each of the first $m$ columns of the Young diagram leaves a self-conjugate partition.

\smallskip\noindent 
In \cite{AB19}, the authors give a Beck-type companion identity for
$$p(n\mid \mbox{distinct odd parts})=p(n\mid \mbox{self-conjugate}),$$ 
and the work of this section has many similarities with \cite{AB19}.

\smallskip
\noindent
Denote by $s_{c'}(n)$ the number of parts of all partitions in $sy\mathcal P_0(n)$  and by $s_e(n)$ the number of parts of all partitions of $n$ into distinct even parts. Before we introduce our first Beck-type identity, we need a definition. Recall that  the rank of a partition $\lambda$ is the difference, $\lambda_1-\ell(\lambda)$, between the largest part in $\lambda$ and the length of $\lambda$. In \cite{BG}, the $M_2$-rank of a partition $\lambda$ is defined as the difference between the largest part and the number of parts  in  the $\text{mod } 2$ diagram of $\lambda$, that is, $$M_2(\lambda)=\left\lceil\frac{\lambda_1}{2}\right\rceil-\ell(\lambda).$$

\begin{theorem}\label{first beck} For  all $n > 0$,  we have $s_{c'}(n)-s_e(n)$ equals twice the number of partitions  into even parts with exactly one even part repeated plus the number of partitions into distinct even parts with even $M_2$-rank. 
\end{theorem}

\begin{proof}
We use the notation $D f(z,q)$ to mean the derivative of $f(z,q)$ with respect to $z$ evaluated at $z=1$, i.e., 
$$Df(z,q):=\left.\left(\frac{\partial}{\partial z}f(z,q)\right)\right|_{z=1}.$$
Note that if $f(z,q)$ is a partition generating function wherein the exponent of $q$ keeps track of the number being partitioned and the exponent of $z$  is  the number of parts, then $Df(z,q)$ is the generating function for the number of parts in the partitions considered.

\noindent
We denote the generating functions for $s_{c'}(n)$  by $S_{c'}(q)$.  Thus,  
$$S_{c'}(q)=\sum_{n\geq 0} s_{c'}(n)q^n .$$

\noindent
To obtain the generating function for the number of partitions in $sy\mathcal P_0$ wherein the exponent of $z$ keeping track of the number of parts, we use the bivariate generating function for self-conjugate partitions \cite{AB19} given by 
$$ \sum_{m\geq 0}\frac{z^{m}q^{m^2}}{(zq^2;q^2)_m}.$$
Given a partition in $sy\mathcal P_0$ with Durfee square of side length $m$, we remove one box from each of the first $m$ columns of the Young diagram to  obtain a self-conjugate partition. Thus, the bivariate generating function for partitions in $sy\mathcal P_0$ wherein the exponent of $z$ keeping track of the number of parts is
 \begin{equation}\label{bivp_0}F(q;z):=\sum_{m\geq 1}\frac{z^{m+1}q^{m^2+m}}{(zq^2;q^2)_m}.\end{equation}

\smallskip
\noindent 
We begin by writing $F(q;z)$ as a limit. We mean  
$$F(q;z)=\lim_{\tau\to 0}z\sum_{m\geq 1}\frac{(-\frac{q^2}{\tau};q^2)_m(q^2;q^2)_m z^m\tau^m}{(q^2;q^2)_m(zq^2;q^2)_m}.$$ 
Next, we apply the transformation on the last line of pg. 38 of \cite{A98} in which we first replace  $q$  by $q^2$ and then substitute $a= -\frac{q^2}{\tau}, b=q^2, c=zq^2, t=z\tau$. We follow this through the limit as $\tau \to 0$. We obtain 

\begin{equation}F(q;z)=-z+ z(1-z)\sum_{m\geq 0}(-q^2; q^2)_m z^m.
\end{equation} 

\noindent 
Applying the operator $D$, we obtain 
\begin{align*} S_{c'}(q)=&D F(q;z) \\
=&-1+ \lim_{z\rightarrow 1^{-}}(1-z)\sum_{m\geq0}(-q^2;q^2)_mz^m+D(1-z)\sum_{m\geq 0}(-q^2; q^2)_m z^m \\
=&-1+\sum_{m\geq0}\frac{q^{m^2+m}}{(q^2;q^2)_m}+D(1-z)\sum_{m\geq 0}(-q^2; q^2)_m z^m \\
=&-1+(-q^2;q^2)_{\infty}++D(1-z)\sum_{m\geq 0}(-q^2; q^2)_m z^m.
\end{align*}

\noindent 
To find  $D(1-z)\sum_{m\geq 0}(-q^2; q^2)_m z^m$ we use \cite[Proposition 2.1 and Theorem 1]{AJO01}. When using Theorem 1 in \cite{AJO01}, we first replace $q$ by $q^2$ and then set $a=0$ and $t=-q^2$. Thus, 
\begin{align*} D(1-z)\sum_{m\geq 0}(-q^2; q^2)_m z^m &=\sum_{n\geq0}((-q^2;q^2)_{\infty}-(-q^2;q^2)_n) \\
&=\frac12\sum_{n\geq1}\frac{q^{n(n-1)}}{(-q^2;q^2)_{n-1}}+(-q^2;q^2)_{\infty}\left(-\frac12+\sum_{m\geq1}\frac{q^{2m}}{1-q^{2m}}\right).
\end{align*}
Therefore, the generating function for the number of parts of all partitions in $sy\mathcal P_0(n)$ can be written as 

\newpage

\begin{align*} S_{c'}(q)&=D F(q;z) \\ 
&= -1+\frac{1}{2} \sum_{n=1}^\infty \frac{q^{n(n-1)}}{(-q^2;q^2)_{n-1}}+(-q^2;q^2)_\infty \left(\frac12+\sum_{m\geq 1}\frac{q^{2m}}{1-q^{2m}}\right).
\end{align*}
We notice that
$$\sum_{n\geq1}\frac{q^{n(n-1)}}{(-q^2;q^2)_{n-1}}=\sigma(q^2),$$
where $\sigma(q^2)$ is the expression (1.1) in \cite{ADH88}. There, it is shown that $\sigma(q)=\sum_{n=0}^\infty S(n)q^n$, where $S(n)$ counts the number of partitions of $n$ into distinct parts with even rank minus  the number of partitions of $n$ into distinct parts with odd rank. Thus, $\sigma(q^2)=\sum_{n\geq0}\widetilde{S}(n)q^n$, where $\widetilde{S}(n)$ is the number of partitions of $n$ into distinct even parts with even $M_2$-rank minus the number of partitions of $n$ into distinct even parts with odd $M_2$-rank. Then, \begin{equation}\label{m2-rank id} \frac{1}{2} \left(\sum_{n=1}^\infty \frac{q^{n(n-1)}}{(-q^2;q^2)_{n-1}}+(-q^2;q^2)_\infty\right)=p(n\mid \text{distinct even parts, even $M_2$-rank}).
\end{equation}

\smallskip
\noindent 
Next, we observe that $(-q^2;q^2)_\infty\sum_{m\geq 1}\frac{q^{2m}}{1-q^{2m}}$ is the generating function for the cardinality of 
$$\mathcal E(n)= \{(\lambda, (2m)^k)\mid \lambda \text{ has distinct even parts, }m,k\geq 1,|\lambda|+2mk=n \}.$$ Clearly, the subset of $\mathcal E(n)$ consisting of pairs $(\lambda, (2m))$, where $2m$ is not a part of  $\lambda$, is in one-to-one correspondence with the set of parts of all partitions of $n$ into distinct even parts. To see this, notice that if $\mu$ is a partition into distinct even parts, for each part $2t$ in $\mu$, we can create a pair $(\mu\setminus (2t), (2t))\in \mathcal E(n)$. Here and throughout, if $\eta$ is a partition whose parts (with multiplicity) are also parts of $\mu$, then $\mu\setminus \eta$  stands for the partition obtained from $\mu$ by removing all parts of $\eta$ (with multiplicity).

\smallskip
\noindent 
Now, consider the remaining pairs of partitions $(\lambda, (2m)^k)$ in $\mathcal E(n)$, i.e, pairs with $k\geq 2$ or pairs with $k=1$ and $2m$ is a part of $\lambda$. For each such pair  $(\lambda, (2m)^k)$, we create the partition $\mu=\lambda \cup (2m)^k$. Then, $\mu$ is a partition into even parts with exactly one even part repeated. Each such partition is obtained twice: if $2t$ is the repeated part of $\mu$ and it appears with multiplicity $b\geq 2$, then $\mu$ is obtained from $(\mu\setminus (2t)^b, (2t)^b)$ and also from $(\mu\setminus (2t)^{b-1}, (2t)^{b-1})$. 

\smallskip
\noindent This completes the proof of the theorem. 
\end{proof}

\smallskip
\noindent  
We proceed to derive a weighted Beck-type companion identity for  \eqref{dist-parts}. 

\begin{theorem}\label{second beck} For $n > 0$, we have $s_{c'}(n)-2s_e(n)$ equals the number of partitions partitions into even parts with exactly one even part repeated and if the repeated part is not the smallest, the partition is counted with weight $2$. 
\end{theorem}

\begin{proof} In addition to the notation introduced in the proof of Theorem \ref{first beck}, we denote by $S_e(q)$
 the generating function for  $s_e(n)$.  Thus,  
$$S_e(q)=\sum_{n\geq 0} s_e(n)q^n.$$ 
The generating function for the number of partitions into distinct even parts, where the exponent of $z$ keeping track of the number of parts, is 
\begin{equation}\label{biv_even} E(q;z):=\sum_{m\geq 0}\frac{z^mq^{m^2+m}}{(q^2;q^2)_m}.\end{equation} 
To see this, given a partition $\lambda$ with distinct even parts, remove  an even staircase, i.e., remove parts $2, 4, \ldots$ from each part of $\lambda$ starting with the smallest. We are left with a partition into even parts. In $\lambda$ there are as many parts as the height of the even staircase removed. 

 \smallskip
\noindent 
Hence, using \eqref{bivp_0} and \eqref{biv_even}, we have 
\begin{align} S_{c'}(q)-S_e(q) 
& = D\left(\sum_{m\geq 1}\frac{z^{m+1}q^{m^2+m}}{(zq^2;q^2)_m}-\sum_{m\geq 0}\frac{z^mq^{m^2+m}}{(q^2;q^2)_m} \right)\nonumber \\ 
& =\sum_{m\geq 1}\left(\frac{(m+1)q^{m^2+m}}{(q^2;q^2)_m}+ D \frac{q^{m^2+m}}{(zq^2;q^2)_m}-\frac{mq^{m^2+m}}{(q^2;q^2)_m}\right)\nonumber\\ 
& = \sum_{m\geq 1}\frac{q^{m^2+m}}{(q^2;q^2)_m} + D \sum_{m\geq 1}\frac{q^{m^2+m}}{(zq^2;q^2)_m}.\label{gf}
 \end{align}
 Now 
\begin{align*}\sum_{m\geq 0}\frac{q^{m^2+m}}{(zq^2;q^2)_m}& = \lim_{\tau\to 0}\sum_{m\geq 0}\frac{(-\frac{q^2}{\tau};q^2)_m(q^2;q^2)_m\tau^m}{(q^2;q^2)_m(zq^2;q^2)_m}\\ & = \lim_{\tau\to 0} \frac{(q^2;q^2)_\infty(-q^2;q^2)_\infty}{(zq^2;q^2)_\infty(\tau;q^2)_\infty}\sum_{m\geq 0}\frac{(z;q^2)_m(\tau;q^2)_mq^{2m}}{(q^2;q^2)_m(-q^2;q^2)_m}.
  \end{align*} 
The last equality was obtained from Heine's transformation  \cite[p.19, Cor. 2.3]{A98} by first replacing  $q$  with $q^2$ and then substituting 
$a= -\frac{q^2}{\tau}, b=q^2, c=zq^2, t=\tau$. 
 
\smallskip
\noindent
 Therefore
 \begin{align} \sum_{m\geq 0}\frac{q^{m^2+m}}{(zq^2;q^2)_m} 
& = \frac{(q^2;q^2)_\infty(-q^2;q^2)_\infty}{(zq^2;q^2)_\infty}  \label{S_c-S_o}\\ & \qquad + \frac{(q^2;q^2)_\infty(-q^2;q^2)_\infty}{(zq^2;q^2)_\infty}(1-z)\sum_{m\geq 1}\frac{(zq^2;q^2)_{m-1}q^{2m}}{(q^2;q^2)_m(-q^2;q^2)_m}. \nonumber 
\end{align}
We apply $D$ to \eqref{S_c-S_o} and use  \eqref{gf} to obtain  
\begin{align} S_{c'}(q)- S_e(q)   \nonumber
&=  \sum_{m\geq 1}\frac{q^{m^2+m}}{(q^2;q^2)_m}    +\sum_{m\geq 1}\frac{ (-q^2;q^2)_\infty q^{2m}}{1-q^{2m}}  - \sum_{m\geq 1}\frac{(-q^2;q^2)_\infty q^{2m}}{(1-q^{2m})(-q^2; q^2)_m}\nonumber\\
 &  = \sum_{m\geq 1}\frac{q^{m^2+m}}{(q^2;q^2)_m}  + \sum_{m\geq 1}\frac{(-q^2;q^2)_\infty q^{2m}}{1-q^{2m}} - \sum_{m\geq 1}\frac{q^{2m}(-q^{2m+1};q^2)_\infty}{1-q^{2m}}.\label{gfs}
\end{align}
The first expression in \eqref{gfs} is the generating function for partitions with distinct even parts. As in the proof of Theorem \ref{first beck}, the second expression is the generating function for $\mathcal E(n)$, i.e., pairs of partitions $(\lambda, (2m)^k)$, where $\lambda$ has distinct even parts, $m,k\geq 1$ and $|\lambda|+ 2mk=n$.  Finally, the last expression  generates partitions into even parts in which the smallest part may be repeated. These partitions correspond to pairs $(\lambda, (2m)^k)\in \mathcal E(n)$, such that the parts of $\lambda$ are greater than $2m$. 

\smallskip
\noindent
Then, the difference of the last two expressions is the generating function for 
$$\mathcal E'(n):=\{(\lambda, (2m)^k) \in \mathcal E(n) \mid \lambda \text{ has parts of size at most $2m$}\}.$$ 
As in the proof of Theorem \ref{first beck}, the subset of $\mathcal E'(n)$ consisting of pairs $(\lambda, (2m))$ such that $2m$ is not a part of  $\lambda$ is in one-to-one correspondence with the set of  parts that are not smallest in all partitions of $n$ into distinct even parts. We can view the set of partitions of $n$ into distinct even parts (generated by the first sum of \eqref{gfs}) as corresponding to the set of smallest parts in all partitions of $n$ into distinct even parts.

\smallskip
\noindent Next, we consider the remaining pairs of partitions $(\lambda, (2m)^k)$ in $\mathcal E'(n)$, i.e, pairs with $k\geq 2$ or pairs with $k=1$ and $2m$ is a part of $\lambda$. For each such pair  $(\lambda, (2m)^k)$, we create the partition $\mu=\lambda \cup (2m)^k$. Then, $\mu$ is a partition into even parts with exactly one even part repeated. If the repeated part of $\mu$, $2t$ is the smallest, the partition $\mu$ is obtained exactly once: from $(\mu\setminus {2t}^{b-1})$, where $b$ is the multiplicity of $2t$ in $\mu$. If the repeated part is not the smallest, each such partition is obtained twice: if $2t$ is the repeated part of $\mu$ and it appears with multiplicity $b\geq 2$, then $\mu$ is obtained from $(\mu\setminus (2t)^b, (2t)^b)$ and also from $(\mu\setminus (2t)^{b-1}, (2t)^{b-1})$. 

\smallskip
\noindent This completes the proof of the theorem. 
\end{proof}

\smallskip
\noindent

\section {On the parity of  $p(n \mid \text{distinct parts, odd/even rank})$} \label{odd-even rank}

\noindent
As mentioned in the previous section,  the generating function for the number of partitions of $n$ into distinct parts with even rank minus the number of partitions of $n$ into distinct parts with odd rank is given by  \cite{ADH88}
$$\sigma(q)=\sum_{n\geq0}\frac{q^{\binom{n+1}2}}{(-q;q)_n}.$$ Analogous to identity \eqref{m2-rank id}, we have 
\begin{equation}\label{rank id} \frac{1}{2} \left(\sum_{n=0}^\infty \frac{q^{\binom{n+1}2}}{(-q;q)_{n}}+(-q;q)_\infty\right)=p(n\mid \text{distinct  parts, even rank}).
\end{equation}
Computed modulo $2$, the identity \eqref{rank id} together with the pentagonal number theorem implies 
\begin{equation}\label{PNT} \sum_{n\geq0}\frac{q^{\binom{n+1}2}}{(-q;q)_n}\equiv (-q;q)_{\infty}\equiv \sum_{m\in\mathbb{Z}}q^{\frac{m(3m-1)}2}.\end{equation}
In this context, we consider the generating function \cite{ADH88} for the number $g(n,r)$ of distinct partitions of $n$ having rank $r$:
$$G(z,q):=\sum_{n\geq0}\sum_{r\geq 0} g(n,r)z^rq^n=\sum_{n\geq0}\frac{q^{\binom{n+1}2}}{(zq;q)_n}.$$
Thus, $\sigma(q)=G(-1,q)$. 

\smallskip
\noindent 
To determine the parity behavior of the number of distinct partitions of $n$ with odd, respectively even rank, we first 
prove a preliminary result.

\begin{lemma} \label{pre-lemma} We have
\begin{align*} \sum_{n\geq0}\frac{q^{\binom{n+1}2}}{(zq;q)_n}=\sum_{n\geq0}q^{\frac{n(3n+1)}2}(1+q^{2n+1})\prod_{j=1}^n\frac{z+q^j}{1-zq^j}.
\end{align*}
\end{lemma}
\begin{proof} In the Rogers-Fine identity \cite[p. 233, eq. (9.1.1)]{AB05}, replace $\alpha=-\frac{q}{\tau}, \beta=zq$, and let $\tau\rightarrow0$. Simplification yields the desired result.
\end{proof}

\begin{theorem} \label{odd-odd-odd-even} Let  $n$ be a positive integer.  Then, 
\begin{itemize}\item[(i)] $p(n \mid \text{distinct parts, odd rank})$ is odd  if and only if $n=\frac{k(3k-(-1)^k)}2$ for some positive integer $k$;
 \item[(ii)] $p(n \mid \text{distinct parts, even rank})$ is odd  if and only if $n=\frac{k(3k+(-1)^k)}2$ for some positive integer $k$.\end{itemize}
\end{theorem}

\begin{proof} As usual, operator $D$ computes the derivative $\frac{d}{dz}$ and evaluates at $z=1$. All  congruences in this proof are {\it modulo $2$}.  It is straightforward to check that, expanded as a series, $D\frac{z+x}{1-zx}\equiv1$ and  $\frac{1+x}{1-x}\equiv 1$,  i.e., all coefficients of powers of $x$ except the constant term are even. Using these facts and  Lemma \ref{pre-lemma}, we have
\begin{align} DG(z,q)&=\sum_{n\geq0}q^{\frac{n(3n+1)}2}(1+q^{2n+1})\sum_{j=1}^n\prod_{\substack{i=1 \\ i\neq j}}^n\frac{1+q^i}{1-q^i}D\frac{z+q^j}{1-zq^j}\label{der} \\
&\equiv \sum_{n\geq0}q^{\frac{n(3n+1)}2}(1+q^{2n+1})\sum_{j=1}^n1  \nonumber \\
&\equiv \sum_{\substack{n\geq0 \\ \text{$n$ odd}}} q^{\frac{n(3n+1)}2}(1+q^{2n+1})=\sum_{n\geq0}q^{(2n+1)(3n+2)}(1+q^{4n+3}), \nonumber
\end{align}
which is equivalent to the first assertion of the theorem.

\smallskip
\noindent 
To prove the second assertion, we calculate $D\,zG=G(1,q)+DG(z,q)$ by making use of 
(\ref{PNT}) and \eqref{der}
\begin{align*} D\,zG(z,q)&=\sum_{n\geq1}\frac{q^{\binom{n+1}2}}{(q;q)_n}+\sum_{n\geq0}q^{\frac{n(3n+1)}2}(1+q^{2n+1})\sum_{j=1}^n\prod_{\substack{i=1 \\ i\neq j}}^n\frac{1+q^i}{1-q^i}D\frac{z+q^j}{1-zq^j}  \\
&\equiv \sum_{m\in\mathbb{Z}^{*}}q^{\frac{m(3m-1)}2}+ \sum_{k\geq1}q^{\frac{k(3k-(-1)^k)}2} \\
&\equiv \sum_{k\geq1}q^{\frac{k(3k+(-1)^k)}2},
\end{align*}
where $\mathbb{Z}^{*}$ denotes the set of non-zero integers. The proof is complete.
\end{proof}

\smallskip
\noindent
The preceding discussion leads to the following  $q$-series identity.

\begin{corollary} We have
$$\sum_{n\geq1}\frac{q^{\binom{n+1}2}}{(q;q)_n}\sum_{m=1}^n\frac{q^m}{1-q^m}
=\sum_{n\geq1}\frac{q^{\frac{n(3n+1)}2}(1+q^{2n+1})\,(-q;q)_n}{(q;q)_n}\sum_{j=1}^n\frac{1+q^{2j}}{1-q^{2j}}.$$
\end{corollary}
\begin{proof} Both sides of the identity equal $D\, G(z,q)$. 
The left-hand side is obtained by  logarithmic-differentiation, while the right-hand side is obtained from \eqref{der}.
\end{proof}

\section{Binary inversion sums and sums of hook lengths} \label{bi-hooks}

\noindent
 A pair $i<j$ is an {\it inversion} of a binary string $w=w_1\cdots w_n$ if $w_i=1>0=w_j$. The {\it binary inversion sum} of $w$ is given by the sum
$s(w)=\sum(j-i)$ over all inversions of $w$. Given a binary string (or word)   $w=w_1\cdots w_n$, the {\it position} of $w_i$  in $w$ is $i$. 

\smallskip
\noindent
Given a partition $\lambda$, its {\it bit string},  $b(\lambda)$, is a binary word starting with $1$ and ending with $0$ defined as follows. Start at SW corner of the Young diagram  and travel along the outer profile going NE. For each step to the right, record a $1$; for each step up, record a $0$. By the {\it length} of $b(\lambda)$ we mean the number of digits in $b(\lambda)$. For instance, for the partition $\lambda= (5,3,3,2,1)$ of Example \ref{eg1} we get the bit string $b(\lambda)=1010100110$ and the length of $b(\lambda)$ is $10$.

\smallskip
\noindent Before we state the main result of this section, we give a motivating example. 
\begin{example} We list the partitions of $n=2, 3$, and $4$. For each partition $\lambda$, we give the bit string $b(\lambda)$,  the binary inversion sum $s(b(\lambda))$,  the list $H_\lambda$ of hook lengths of cells of $\lambda$, and the sum of all hook lengths of $\lambda$, i.e.,  $\sum_{u\in\lambda}h^{\lambda}(u)$.  Notice that we give $H_{\lambda}$ as a list of sub-lists of hook lengths along rows of the Young diagram.

\begin{center}
\begin{tabular}{ |c|c|c|c|c| } 
\hline
partition $\lambda$ & $b(\lambda)$ & $s(b(\lambda))$ & $H_{\lambda}$ & sum of elements in $H_{\lambda}$ \\
\hline
(2) & 110 & 3 & [(2,1)] & 3 \\
(1,1) &100 & 3 & [(2),(1)]& 3 \\
(3) & 1110 & 6 & [(3,2,1)] & 6 \\
(2,1) & 1010 & 5 &[(3,1),(1)] & 5 \\
(1,1,1) & 1000 & 6 & [(3),(2),(1)] & 6 \\
(4) & 11110 & 10 & [(4,3,2,1)]  & 10 \\
(3,1) & 10110 & 8 & [(4,2,1),(1)] & 8 \\
(2,2) & 1100 & 8 & [(3,2),(2,1)]  & 8 \\
(2,1,1) & 10010 & 8 & [(4,1),(2),(1)] & 8 \\
(1,1,1,1) & 10000 & 10 & [(4),(3),(2),(1)] & 10 \\
\hline
\end{tabular}
\end{center}
\end{example}

\begin{theorem} Given any partition $\lambda$, the sum of its hooks, $\sum_{u\in\lambda}h^{\lambda}(u)$, equals the inversion sum, $s(b(\lambda))$, of its binary bit $b(\lambda)$.
\end{theorem}

\begin{proof} Let $\lambda$ be a partition of $n$ and $b(\lambda)$ its bit string.  The number of parts $\ell(\lambda)$ of $\lambda$ equals the number of $0$ digits in $b(\lambda)$. The perimeter of $\lambda$, which is defined as the hook-length $h^{\lambda}(1,1)$, is equal to the length of $b(\lambda)$ minus $1$. 

\smallskip
\noindent
Each  cell $(a,c)$ in the Young diagram of $\lambda$ defines a sub-diagram consisting of the cell $(a,c)$ and all cells of  the Young diagram of  $\lambda$ weakly to the right and weakly below $(a,c)$. Denote the partition corresponding to this sub-diagram by $\lambda^{(a,c)}$. Then $b(\lambda^{(a,c)})$ is the sub-string of $b(\lambda)$ starting at the $c^{th}$ digit from the left equal to $1$ (we do not count the $0$ digits) and ending after the $a^{th}$ digit from the right equal to $0$ (we do not count the $1$ digits). The length of $b(\lambda^{(a,c)})$ minus $1$ is precisely the hook length $h(a,c)$ in $\lambda$.  Moreover, if the $c^{th}$ digit from the left equal to $1$ is in position $i$ in $b(\lambda)$ and the $a^{th}$ digit from the right equal to $0$  is in position $j$, then $j-i$ is precisely the length of $b(\lambda^{(a,c)})$ minus $1$, i.e., $h(a,c)$. 
Conversely, every inversion $i<j$ determines a sub-string of $b(\lambda)$ that starts with $1$ in position $i$ in $b(\lambda)$  and ends with $0$ in position $j$ in $b(\lambda)$. The length of the sub-string minus $1$ is precisely $j-i$ and is the hook length $h(a,c)$  where $c$ is the number of $1$ digits in positions $\leq i$ and $a$ is the number of $0$ digits in positions $\geq j$ in $b(\lambda)$. Thus, for each inversion $i<j$, the difference $j-i$ is the hook length of a unique cell in the Young diagram of $\lambda$. 
\end{proof}

\section{On the $x$-ray list of a partition} \label{two stats} 

\noindent
Let $\lambda=(\lambda_1,\lambda_2,\dots,\lambda_{\ell(\lambda)})$ be an integer partition of length $\ell(\lambda)$ and designate $m:=\max\{\lambda_1,\ell(\lambda)\}$. We construct an $m\times m$ matrix whose entry in the $i^{th}$ row and $j^{th}$ column is $1$ if $(i,j)$ is a cell in the Young diagram of $\lambda$ and $0$ otherwise. Since the Young diagram of a partition is also known as the Ferrers diagram, the matrix defined here is referred to as the {\it Ferrers matrix} of $\lambda$ on OEIS \cite{FMatrix}. 
\begin{example} \label{eg 3} Given the partition, 
$\lambda=(4,3,1)\vdash 8$, we have  $m=4$ and the corresponding $4\times 4$ matrix is shown below. 
$$\begin{bmatrix} 1&1&1&1 \\ 1&1&1&0 \\ 1&0&0&0 \\ 0&0&0&0
\end{bmatrix}$$
\end{example}
\noindent 
Given a partition $\lambda$, we add the elements of each anti-diagonal in the Ferrers matrix of $\lambda$ to obtain a composition, denoted by $\lambda_x$,  which we call the {\it $x$-ray list} of $\lambda$.  
For the partition $\lambda=(4,3,1)$ in Example \ref{eg 3}, we have $\lambda_x=(1,2,3,2,0,0,0)$. 

\smallskip
\noindent  Since  the number of entries equal to $1$ in an anti-diagonal is invariant under conjugation, we have $\lambda_x=\lambda'_x$  for all partitions $
\lambda$. 
Thus, when studying $x$-ray lists, it is enough to   consider   partitions $\lambda$ satisfying $\lambda_1\geq \ell(\lambda)$. So, the Ferrers matrix is  an $M\times M$ matrix, where $M=\lambda_1$.

\smallskip
\noindent
The first result, below, proves that the $x$-ray lists of partitions are unimodal compositions. In addition, the sub-sequence of an $x$-ray consisting of the initial terms  up to the first occurrence of a peak is strictly increasing.

\begin{lemma} \label{x-ray-lemma} If $\lambda$ is a partition, then $\lambda_x$ is of the form $(1,2,3,\dots,n,a_1,a_2,\dots,a_r)$, where $n\geq a_1\geq a_2\geq\cdots\geq a_r$.
\end{lemma}

\begin{proof} Let $\lambda$ be a partition such that $\lambda_1\geq\ell(\lambda)$ and let $M=\lambda_1$. Consider all initial anti-diagonals with no entries equal to $0$. These contribute the following initial parts in the composition $\lambda_x$:
\begin{itemize}
\item[(i)]$1, 2, 3,\dots, n$ where $n\leq M$, or \medskip

\item[(ii)] $1, 2, 3, \dots, M, M-1, M-2, \dots, M-j$.
\end{itemize}

\smallskip
\noindent
Let us consider the first anti-diagonal containing  a $0$ entry.  In case (i),  the sum of its entries,  $a_1$,  will be at most $n$. In case (ii), the sum of its entries will be strictly less than $M-j$. Furthermore, zeros propagate zeros, i.e., in a column,  below a $0$ there are only  $0$ entries and, in a row, to the right of a $0$ there are only $0$ entries. Thus, for any string of consecutive zeros in a given anti-diagonal, there is a string with at least as many zeros in the next anti-diagonal. This necessitates a non-increase in the sum of entries of the next anti-diagonal.\end{proof}
\begin{example} In the partially filled Ferrers matrices below,

$$\begin{bmatrix}
1&1&1&1&1 \\ 1&1&1&0&\square \\ 1&1&0&\square \\ 1&0&\square \\ 1&\square
\end{bmatrix} \qquad \qquad
\begin{bmatrix}
1&1&1&1&1 \\ 1&1&1&0&\square \\ 1&1&0&\square \\ 1&0&\square \\ 0&\square
\end{bmatrix},$$

\bigskip
\noindent 
the positions marked with $\square$ are forced to be filled with $0$.\end{example}

\begin{theorem} \label{x-ray-thm} The number of different $x$-ray lists  of $n$ equals to the number of partitions of $n$ into distinct parts. \end{theorem}

\begin{proof} We first show that if $\pi=(b_1, b_2, \ldots)$ and $\rho=(c_1, c_2, \ldots)$ are  partition of $n$ into distinct parts such that  $\pi\neq \rho$, then their $x$-ray lists are different, i.e., $\pi_x\neq\rho_x$. Let $i$ be the smallest integer such that $b_i\neq c_i$, say $b_i>c_i$.   In the Ferrers matrix of a partition with distinct parts,  in an anti-diagonal, the entries SW of a $0$ are all equal to $0$. Therefore, the anti-diagonal of $\pi$ containing the end point of part $b_i$, i.e., the  $(i+b_i-1)^{st}$ anti-diagonal,  contains at least one more $1$  than the same anti-diagonal in $\rho$. Then, $\pi_x\neq \rho_x$.
Consequently, the number of $x$-ray lists is at least as large as the number of partitions into distinct parts.

\smallskip
\noindent On the other hand, each $x$-ray list, $\lambda_x$, of $n$ corresponds to a partition $\mu$ of $n$ with distinct parts as follows.  If the $i^{th}$  entry of $\lambda_x$  is $r$, fill in the $i^{th}$ anti-diagonal of a matrix with $r$ entries equal to $1$ (starting in the first row of the matrix) followed by $0$ entries. Let $\phi(\lambda_x)$ be the partition  whose Ferrers matrix  is the matrix obtained above. By Lemma \ref{x-ray-lemma} and the construction of the matrix, $\phi(\lambda_x)$ is a  partition with distinct parts.
Hence, the number of $x$-ray lists is no more than the number of partitions of $n$ into distinct parts. The proof is complete.
\end{proof}

\noindent
\begin{example} We illustrate the proof of Theorem \ref{x-ray-thm} for $n=7$.

\bigskip
\begin{center}
\begin{tabular}{ |c|c|c|c| } 
\hline
$\lambda$ with $\lambda_1\geq\ell(\lambda)$ & $\lambda_x$ & $\phi(\lambda_x)$ \\
\hline
\multirow{1}{4em} {7} & (1,1,1,1,1,1,1) & 7 \\
\hline
\multirow{1}{4em} {61} & (1,2,1,1,1,1) & 61 \\ 
\hline
\multirow{2}{4em}{52 \\ 511} & (1,2,2,1,1) & 52 \\ 
& & \\
\hline
\multirow{2}{4em}{43 \\ 411} & (1,2,2,2) & 43 \\ 
& & \\
\hline
\multirow{3}{4em}{421 \\ 331 \\ 322 } & (1,2,3,1) & 421 \\ 
& & \\
& & \\
\hline
\end{tabular}
\end{center}
\end{example}

\smallskip
\noindent
We now explore a different aspect of $x$-ray lists of partitions.

\smallskip\noindent We say that a partition $\mu$ is contained in partition $\eta$ if $\mu_i\leq \eta_i$ for all $i$ (the partitions are padded with $0$ after the last part). 
Recall that $\delta_r$ denotes the staircase partition of length $r$. We say that $\lambda$ is {\it maximally contained} in $\delta_r$ if it is contained in $\delta_r$ but not in $\delta_{r-1}$. We consider the following motivating example.

\begin{example}
The  table below  \cite{STSP} is constructed as follows. For $n\geq1$, the $n^{th}$ row  records the number of partitions of $n$ that are maximally contained in $\delta_r$ for $1\leq r\leq n$. 
\begin{align*}
&1 \\ &0, 2 \\ &0, 1, 2 \\ &0, 0, 3, 2 \\ &0, 0, 3, 2, 2 \\ &0, 0, 1, 6, 2, 2 \\ &0, 0, 0, 7, 4, 2, 2.
\end{align*}
For example, $\{(4), (3,1), (2,2), (2,1,1), (1,1,1,1)\}$ is the set of  partitions of $n=4$. The fourth row in the above table shows that there are no partitions maximally contained in $\delta_1$ or $\delta_2$;  there are $3$ partitions maximally contained in $\delta_3$; there are $2$ partitions maximally contained in $\delta_4$. Thus, the the fourth row is $0, 0, 3, 2$.

\smallskip
\noindent
On the other hand, the multiset of $x$-ray lists of $n=4$ with zero entries omitted is $\{1111, 211, 211, 211, 1111\}$. There are no $x$-ray lists of length $1$ or $2$; there are $3$ $x$-ray lists of length $3$; there are 2 $x$-ray lists of length $4$. This data is precisely the fourth row in the above table: $0, 0, 3, 2$. \end{example}

\noindent This agreement  is not accidental as shown in the next theorem. By the {\it length} of an $x$-ray list, we mean the number of non-zero entries in the list.

\begin{theorem} \label{biject-triangle} For $n\geq1$ and $1\leq r\leq n$, the number of partitions of $n$ with $x$-ray list of length $r$ equals the number of partitions of $n$ maximally contained in $\delta_r$. \end{theorem}

\begin{proof} Let $\lambda\vdash n$ and suppose its x-ray list, $\lambda_x$, has length $r$. From Lemma \ref{x-ray-lemma}, the first $r$ anti-diagonals in the Young diagram  of $\lambda$ are non-empty while the $(r+1)^{st}$ anti-diagonal is empty. Since $(\delta_r)_x=(1,2,3,\dots,r)$, it follows that $\lambda$ is maximally contained in $\delta_r$.

\smallskip
\noindent
Conversely, if $\lambda$ is maximally contained in $\delta_r$, then the $(r+1)^{st}$ anti-diagonal of $\lambda$ is empty and the $r^{th}$ anti-diagonal is non-empty. Again, by Lemma \ref{x-ray-lemma}, the $x$-ray list $\lambda_x$ has length $r$.
\end{proof}

\end{document}